\documentclass[11pt]{article}
\usepackage{latexsym}
\usepackage{graphicx} 
\usepackage{color} 
\usepackage{amsmath, amsthm, amssymb}
\usepackage{enumerate}
\usepackage{float}
\usepackage{url}
\usepackage{stackrel}
\usepackage{mathrsfs,dsfont}
\usepackage{subfigure}
\usepackage{wrapfig}

\newtheorem{theorem}{Theorem}[section]
\newtheorem{corollary}[theorem]{Corollary}

\newtheorem{lemma}[theorem]{Lemma}
\newtheorem{example}[theorem]{Example}
\newtheorem{notation}[theorem]{Notation}
\newtheorem{definition}[theorem]{Definition}

\newtheorem{remark}[theorem]{Remark}

\newcommand{\be}{\begin{equation}}
\newcommand{\ee}{\end{equation}}
\newcommand{\bd}{\begin{displaymath}}
\newcommand{\ed}{\end{displaymath}}
\newcommand{\beal}{\begin{align}}
\newcommand{\enal}{\end{align}}
\newcommand{\been}{\begin{enumerate}}
\newcommand{\enen}{\end{enumerate}}
\newcommand{\beit}{\begin{itemize}}
\newcommand{\enit}{\end{itemize}}

\def\SS{\mathcal S}

\def\Om{\Omega}

\def\bX{\bar{X}}

\def\bOm{\bar{\Omega}}
\def\bP{\bar{P}}
\def\vp{\varphi}
\def\sig{\sigma}
\def\llf{\left \lfloor}
\def\rrf{\right \rfloor}

\newcommand{\R}{\mathbb{R}}

\newcommand{\Z}{\mathbb{Z}}

\def\ep{\epsilon}
\def\ra{\rightarrow}

\def\ds{\displaystyle}

\providecommand{\abs}[1]{\lvert#1\rvert}

\begin{document}

\title{Order of magnitude time-reversible Markov chains and characterization of clustering processes}
\author{Badal Joshi}

\maketitle

\begin{abstract}
We introduce the notion of order of magnitude reversibility (OM-reversibility) in Markov chains that are parametrized by a positive parameter $\ep$. OM-reversibility is a weaker condition than reversibility, and requires only the knowledge of order of magnitude of the transition probabilities. For an irreducible, OM-reversible Markov chain on a finite state space, we prove that the stationary distribution satisfies order of magnitude detailed balance (analog of detailed balance in reversible Markov chains). The result characterizes the states with positive probability in the limit of the stationary distribution as $\ep \to 0$, which finds an important application in the case of singularly perturbed Markov chains that are reducible for $\ep=0$. We show that OM-reversibility occurs naturally in macroscopic systems, involving many interacting particles.  Clustering is a common phenomenon in biological systems, in which particles or molecules aggregate at one location. We give a simple condition on the transition probabilities in an interacting particle Markov chain that characterizes clustering. We show that such clustering processes are OM-reversible, and we find explicitly the order of magnitude of the stationary distribution. Further, we show that the single pole states, in which all particles are at a single vertex, are the only states with positive probability in the limit of the stationary distribution as the rate of diffusion goes to zero. 

\end{abstract}

\section{Introduction}

This paper has two objectives. The first is to introduce the notion of {\em $\ep$-order of magnitude reversibility} (OM-reversibility) in a family of Markov chains $X^\ep$, parametrized by  $\ep >0$. The condition of OM-reversibility is weaker than reversibility, and requires only the knowledge of order of magnitude of the transition probabilities. The main result in this article (Theorem \ref{thm:omstat}) gives the order of magnitude of the probabilities in the stationary distribution of an irreducible, OM-reversible Markov chain on a finite state space. The order of magnitude of the unique stationary distribution $\pi^\ep$ on $X^\ep$ for $\ep>0$ is sufficient to characterize the set of states with positive probability in the limit of the stationary distribution as $\ep \to 0$. The second objective of this paper is to characterize {\it clustering processes}. Clustering processes are interacting particle systems, in which the particles have a tendency to aggregate at one location (we will refer to the phenomenon of clustering to a single location as {\em pole formation}). For an interacting particle system to be a clustering process, we only require that the probability for a particle to move to an adjacent, unoccupied vertex is an order of magnitude smaller than the probability to move to an adjacent, occupied vertex (see Definition \ref{def:clustprocess}).  We prove that clustering processes are OM-reversible and we explicitly give the order of magnitude of the probabilities in the stationary distribution for all particle configurations (Theorem \ref{thm:clusteringstat}).


At microscopic or atomic scales, time-reversibility or simply reversibility is frequently a fundamental property of physical systems. Mathematically, reversibility is formulated as {\em detailed balance} or as the {\em Kolmogorov  cycle condition}, see Theorem \ref{thm:reversibility}. If a stochastic process is reversible, then a stationary distribution exists and further the detailed balance condition can be solved to give the stationary distribution explicitly. 

At macroscopic scales, such as the ones that occur in cell biology involving multiple interating particles, reversibility is less likely to be encountered. On the other hand, there may be examples of systems that are `approximately reversible' and we might expect that if a stationary distribution exists, then it satisfies a corresponding version of `approximate detailed balance'. In the same spirit, we introduce the notion of {\em order of magnitude reversibility (OM-reversibility)} in Markov chains (see Definition \ref{def:omreversible}), where the order of magnitude is with respect to a positive parameter $\ep$. We prove that the unique function (up to an additive constant) that satisfies the OM-reversibility condition in a finite, irreducible Markov chain is the order of magnitude of the stationary distribution. Furthermore, we show that for an irreducible Markov chain on a finite state space, OM-reversibility is equivalent to a condition that we call {\em order of magnitude Kolmogorov cycle condition} or {\em OM-cycle condition}. 

We illustrate how OM-reversibility can arise in macroscopic systems through an example of a class of interacting particle systems that we will refer to as {\em clustering processes}. The phenomenon of clustering occurs frequently in biological systems. To give an instance from cell biology, cells that are initially spatially symmetric, can spontaneously lose symmetry and evolve into an asymmetric state with molecules clustered together at one spot. Bud formation in a yeast cell is initiated when Cdc42 molecules aggregate at one location on the surface of the cell  \cite{altschuler2008spontaneous, butty2002positive, goryachev2008dynamics, ozbudak2005system}. Besides yeast, hippocampal axons \cite{shi2003hippocampal}, canine kidney cells \cite{gassama2006phosphatidylinositol}, and human chemotaxing neutrophils \cite{weiner2002ptdinsp3} show clustering of specific molecules resulting in cellular polarity. Other examples and models from the biological literature can be found in \cite{drubin1996origins, ebersbach2007exploration, gierer1972theory, irazoqui2003scaffold, turing1952chemical, wedlich2004robust}. A somewhat different example of pole formation is the firing frequency of neurons in a network aggregating to one value, making the population of neurons fire coherently. 

We show that, for the clustering processes, {\em the size of the support}, defined to be the number of occupied vertices in the network, satisfies OM-detailed balance. Thus the size of the support is the order of magnitude of the stationary distribution, up to an additive constant. Of particular interest is the identification of the states that have a positive probability in the stationary distribution in the limit $\ep \to 0$. Consider, for instance, a Markov chain $X^\ep$ which is irreducible for $\ep >0$ but reducible for $\ep=0$. A natural question is, ``Which one of the multiple stationary distributions on $X^0$ is $\lim_{\ep \to 0^+} \pi^\ep$ where $\pi^\ep$ is the unique stationary distribution on $X^\ep$?"  Since single pole states (states with exactly one occupied vertex) minimize the size of the support, only single pole states have positive probability in the stationary distribution in the limit $\ep \to 0$. As another example, we look at clustering processes with carrying capacity, which is a generalization of the clustering processes.



Biologically detailed models of clustering include a model involving a set of coupled partial differential equations \cite{goryachev2008dynamics} as a model for yeast cell pole formation. Altschuler {\it et al.} \cite{altschuler2008spontaneous} propose a model for spontaneous emergence of cell polarity using only the mechanism of positive feedback; detailed mathematical analysis of the model was carried out by Gupta \cite{gupta2010stochastic}. Markov chains parametrized by $\ep$ have been studied in the past using perturbation techniques. Schweitzer \cite{schweitzer1968perturbation} studied the perturbation expansion of the stationary distribution when the Markov chain is irreducible for all values of $\ep \geq 0$. Lasserre \cite{lasserre1994formula} generalized Schweitzer's formula to the case of singularly perturbed Markov chains. Latouche and Louchard \cite{latouche1978return} studied a case of singularly perturbed Markov chains, where the Markov chain is irreducible for $\ep >0$ but is reducible for $\ep=0$ and decomposes into disjoint aggregates of states. Avrachenkov and Haviv \cite{avrachenkov2004first} studied the coefficients of the first terms in the Laurent series of the first return time in the case of singularly perturbed Markov chains. Hassin and Haviv \cite{hassin1992mean} provided a combinatorial algorithm for computing the order of magnitude in $\ep$ of the mean passage time and the first return time in a set of Markov chains parametrized by some $\ep > 0$.  Interacting particle systems on a graph called {\em zero range interaction processes} have been studied in \cite{liggett1973infinite, spitzer1969random, spitzer1970interaction} where a particle jumps from a vertex $x$ to an adjacent vertex $y$ with a probability that depends on the occupancy of $x$. Clustering processes are cousins of zero range interaction processes, because for clustering processes the probability for a particle to jump from $x$ to $y$ depends on the occupancy of $x$ along with the occupancy of all its neighbors including $y$.

This article is organized as follows.  
Section~\ref{section:reversibility} provides an overview of basic results on reversibility in Markov chains and gives two equivalent characterizations of reversibility (Theorem \ref{thm:reversibility}). 
Section~\ref{section:omreversibility} defines order of magnitude reversibility and states the main theorem (Theorem \ref{thm:omstat}) that the stationary distribution on an OM-reversible Markov chain satisfies the order of magnitude detailed balance condition.
Section~\ref{section:clustering} studies the application to clustering and pole formation; Definition \ref{def:clustprocess} provides the definition  and Theorem \ref{thm:clusteringstat} gives the stationary distribution of clustering processes.
Section~\ref{section:carrcap} defines a generalized version of the clustering processes, clustering processes with carrying capacity (Definition \ref{def:clustwithcarr}), and Theorem \ref{thm:carrstat} gives the order of magnitude of the stationary distribution of such processes. 
Section~\ref{section:numsim} provides numerical simulations and observations about the behavior of the Markov chain for small rate of diffusion.

\begin{notation}
Throughout this paper, for $t \in \Z$, $X(t)$ represents a Markov chain, $\Om$ represents the state space of $X$ and $P$ the transition matrix of $X$. We will say that the triple $(X(t),\Om,P)$, or simply $(X,\Om,P)$, is a Markov chain. $\pi$ denotes a stationary distribution on $X$. When we consider a family of Markov chains, parametrized by $\ep \geq 0$, we represent a member of the family as $(X^\ep, \Om, P^\ep)$. When a  stationary distribution exists, it will be denoted as $\pi^\ep$. 
\end{notation}

\section{An overview of reversibility} \label{section:reversibility}

We begin with a brief discussion of the concept of {\em time-reversibility}, often known simply as {\em reversibility}. Outside of the stochastic process setting, time-reversibility plays an important role in many fundamental laws of physics at the microscopic scale. In chemical reaction network theory, microscopic reversibility gives rise to the important idea of {\em detailed balance} \cite{ feinberg1989necessary,lewis1925new,  onsager1931reciprocal, wigner1954derivations}. Casimir extended the idea of detailed balance to electric networks \cite{casimir1949some}. Within the field of Markov chains, there are a number of applications of reversibility, many of which are studied in \cite{kelly1979reversibility}. In the next theorem, we state two equivalent ways of characterizing reversibility, detailed balance and the Kolmogorov cycle condition.

\begin{theorem} (Time reversibility \cite{durrett2010probability, kelly1979reversibility})  \label{thm:reversibility}
For an irreducible Markov chain $(X,\Om,P)$, the {\em detailed balance condition} is equivalent to the {\em Kolmogorov cycle condition}. In other words, the following are equivalent.
\been
\item There exists a function on the state space $\pi : \Om \to \R_{\geq 0}$ satisfying the {\em detailed balance} condition
\begin{align} \label{eq:detbal}
\pi(x) P(x,y) = \pi(y) P(y,x) \mbox{ for all } x,y \in \Om.
\end{align}

\item For every finite sequence of states $(x_1,x_2,\ldots,x_{n-1},x_n=x_1) \subset \Om$, the following {\em Kolmogorov cycle condition} holds
\begin{align} \label{kolmogorov}
\prod_{i=1}^{n-1} P(x_i,x_{i+1}) = \prod_{i=1}^{n-1} P(x_{i+1},x_{i}).
\end{align}

\enen

If either of these conditions is satisfied, we say that $(X,\Om,P)$ is {\em reversible}.
\end{theorem}

 Defining the notion of {\em probability flux from state $x$ to state $y$} as $\pi(x)P(x,y)$, the condition for $\pi$ to be a stationary distribution is simply that the flux into state $x$ {\it i.e.} $\sum_y \pi(y) P(y,x)$ and the flux out of state $x$ {\it i.e.} $\sum_y \pi(x)P(x,y) = \pi(x)$ are equal. The detailed balance condition is a stronger condition that requires that for any two states $x$ and $y$, the flux from state $x$ to state $y$ is equal to the reverse flux from $y$ to $x$. The advantage of reversibility is that it guarantees existence of a stationary distribution in a Markov chain and allows explicitly identifying it.

In this paper we relax the condition of reversibility, and define the weaker notion of {\em order of magnitude reversibility}. Every reversible Markov chain is order of magnitude reversible, but the converse is not true. As we show in Theorem \ref{thm:omstat}, order of magnitude reversibility is sufficient to give orders of magnitudes of the probabilities of states in the stationary distribution. 


\section{Order of magnitude reversibility} \label{section:omreversibility}

The main result in this article is that if $\pi$ is a stationary distribution on an OM-reversible Markov chain (Definition \ref{def:omreversible}), then the order of magnitude of the stationary distribution $\pi$ satisfies a simple additive identity \eqref{omstat} that is analogous to detailed balance in reversible Markov chains (Theorem \ref{thm:omstat}), and we refer to the identity as {\em order of magnitude detailed balance} or simply as {\em OM-detailed balance}. We also show in this section that for an irreducible Markov chain over a finite state space, OM-detailed balance is equivalent to {\em order of magnitude Kolmogorov cycle condition} (Theorem \ref{thm:omkolm}). 


\subsection{Preliminaries}

A function $f: \mathbb{R}_+ \rightarrow \mathbb{R}_+$ is called $\Theta(\ep^k)$ for some $k \in \mathbb{Z}$ if there exist $M_1,M_2, \bar \ep \in \mathbb{R}_+$ 
such that for all $\ep$ with $0<\ep<\bar \ep$, $M_1\ep^k \leq f(\ep) \leq M_2 \ep^k$ holds.  Define the {\em order of magnitude function} $\vp_\ep: {\R_+}^{\R_+} \ra \Z$ by $\vp_\ep(f)=k$ if and only if $f$ is $\Theta(\ep^k)$.   


When it is clear from the context, we will drop the subscript $\ep$ from the order of magnitude function and simply write $\vp$ for $\vp_\ep$. 


\begin{lemma}
The order of magnitude function satisfies the following properties
\been
\item $\vp(f_1+f_2) = \min\{\vp(f_1),\vp(f_2)\}$.
\item $\vp(f_1 \cdot f_2) = \vp(f_1)+\vp(f_2)$.
\enen
\end{lemma}

\begin{proof}
Both properties follow easily from the definition of $\vp$. 
\end{proof}


\begin{definition} \label{def:omreversible}
We say that the Markov chain $(X, \Om, P)$ is {\em order of magnitude reversible} if there exists an integer valued function $\nu: \Om \to \Z$ such that for  all $u,v \in \Om$ with $P(u,v)>0$, $\nu$ satisfies the following {\em order of magnitude detailed balance} condition
\begin{align}\label{omreversible}
\nu(u) +  \vp(P(u,v)) = \nu(v) + \vp(P(v,u)).
\end{align}
\end{definition}

\begin{remark}
It is implicit in the definition that $\ds P(u,v) >0$ if and only if  $\ds P(v,u)>0$.
\end{remark}

Defining $\phi := \vp \circ P$, the reversibility condition \eqref{omreversible} can be rewritten as 
\begin{align} \label{omrevplus}
\nu(u) +  \phi(u,v) = \nu(v) + \phi(v,u)
\end{align}

We will refer to order of magnitude reversibility as {\em OM-reversibility} or alternatively as {\em OM-detailed balance} from here on. 

\begin{theorem}
If an irreducible Markov chain is reversible, then it is OM-reversible. 
\end{theorem}

\begin{proof}
For $x,y \in \Om$ such that $P(x,y)>0$, the detailed balance condition \eqref{eq:detbal} holds. Applying $\vp$, it is evident that $\nu = \vp(\pi)$ satisfies \eqref{omreversible}.
\end{proof}

\begin{lemma} \label{lem:uniqueness}
For an irreducible, OM-reversible Markov chain $(X,\Om,P)$, if there exist $\nu_1$ and $\nu_2$ satisfying \eqref{omreversible}, then $\nu_1 - \nu_2 \equiv c$, a constant. 
\end{lemma}

\begin{proof}
 Let $\nu_1, \nu_2 : \Om \to \Z$ be two functions satisfying OM-reversibility \eqref{omreversible} and let $u,v \in \Om$. By irreducibility, there exists $(u_1,u_2,\ldots,u_{n-1}) \subset \Om$ such that $\prod_{i=0}^{n-1} P(u_i,u_{i+1}) >0$ where $u_0 :=u$ and $u_n :=v$. Then for all $i \in \{0,1,\ldots, n-1\}$ and $j \in \{1,2\}$,  $\nu_j(u_i) + \vp(P(u_i,u_{i+1})) = \nu_j(u_{i+1}) + \vp(P(u_{i+1},u_{i}))$. So that $\nu_1(u_i) - \nu_1 (u_{i+1}) = \nu_2(u_i) - \nu_2 (u_{i+1})$. Summing over  $i \in \{0,1,\ldots, n-1\}$, we get  $\nu_1(u) - \nu_2(u) = \nu_1(v) - \nu_2(v)$. So that $\nu_1 - \nu_2 \equiv c$, a constant.
\end{proof}


\subsection{Stationary distribution on an OM-reversible Markov chain}

We now state our main theorem about the order of magnitude of the stationary distribution on an OM-reversible Markov chain.


\begin{theorem} \label{thm:omstat} (Stationary distribution satisfies OM-detailed balance)
If $\pi$ is a stationary distribution on a finite, irreducible, OM-reversible Markov chain $(X,\Om,P)$ then for all $x,y \in \Om$ such that $P(x,y) >0$, $\pi$ satisfies OM-detailed balance {\it i.e.}
\begin{align} \label{omstat}
\vp(\pi(x)) +  \phi(x,y) = \vp(\pi(y)) + \phi(y,x)
\end{align}
where $\phi = \vp \circ P$. 
\end{theorem}

\begin{proof}
Let $\nu:\Om \to \Z$ be an integer-valued function satisfying \eqref{omreversible} such that $\min_{u \in \Om} \nu(u) =0$. Clearly, $\nu$ is unique by Lemma \ref{lem:uniqueness}. Let $\sig : \Om \to \Z$ where 
\bd
\sig(x):=\vp(\pi(x)) - \nu(x).
\ed 
In order to prove the theorem, we will show that $\sig \equiv 0$. Let $\SS \subset \Om$, and $\SS^c := \Om \setminus \SS$. Let $\partial \SS := \{x \in \SS | P(x,y) > 0 \mbox{ for some } y \in \SS^c \}$. The stationary distribution on any Markov chain satisfies the following identity (see for instance \cite{kelly1979reversibility})
\bd \sum_{x \in \partial \SS} \sum_{y \in \partial \SS^c} \pi(x)P(x,y) = \sum_{x \in \partial \SS} \sum_{y \in \partial \SS^c} \pi(y)P(y,x).
\ed
Noting that $\pi >0$ for a finite, irreducible Markov chain, compose with $\vp$ and use $\vp(\pi(x)) = \sig(x) + \nu(x)$ to get 
\bd
\min_{x \in \partial \SS} \min_{y \in \partial \SS^c} (\sig(x)+ \nu(x)+\phi(x,y)) = \min_{x \in \partial \SS} \min_{y \in \partial \SS^c} (\sig(y)+ \nu(y)+\phi(y,x)).
\ed
Using OM-reversibility, we write $\nu(y)+\phi(y,x)$ as $\nu(x)+\phi(x,y)$ on the right hand side, which gives
\begin{align} \label{sig2}
\min_{x \in \partial \SS} \min_{y \in \partial \SS^c} (\sig(x)+ \nu(x)+\phi(x,y)) = \min_{x \in \partial \SS} \min_{y \in \partial \SS^c} (\sig(y)+ \nu(x)+\phi(x,y)).
\end{align}
Suppose by way of contradiction that $\sig \not \equiv c$ where $c \in \Z$ is a constant. Let $m = \min_{x \in \Om} \sig(x)$ and let $\SS = \{x \in \Om | \sig(x) = m \}$. $\SS$ is non-empty by construction and since $\sig$ is non-constant, $\SS^c$ is also non-empty. 
\begin{align*}
m + \min_{x \in \partial \SS} \min_{y \in \partial \SS^c} (\nu(x)+\phi(x,y)) &= \min_{x \in \partial \SS} \min_{y \in \partial \SS^c} (\sig(x)+ \nu(x)+\phi(x,y)) \\
&= \min_{x \in \partial \SS} \min_{y \in \partial \SS^c} (\sig(y)+ \nu(x)+\phi(x,y)) ~~~\mbox{ (by \eqref{sig2}) } \\
& \geq \min_{y \in \partial \SS^c} \sig(y) + \min_{x \in \partial \SS} \min_{y \in \partial \SS^c} (\nu(x)+\phi(x,y)) \\
 & > m+  \min_{x \in \partial \SS} \min_{y \in \partial \SS^c} (\nu(x)+\phi(x,y)) 
\end{align*}
which implies the contradiction, $m > m$. So the assumption that $\sig$ is non-constant must be false; implying that $\sig \equiv c$. Finally, 
$0 = \min_{x \in \Om}\vp(\pi(x)) = c + \min_{x \in \Om} \nu(x) = c+0=c$, and so $\sig \equiv 0$, which completes the proof. 

\end{proof}



\subsection{Order of magnitude Kolmogorov cycle condition}

Analogous to reversible Markov chains, we define {\em order of magnitude Kolmogorov cycle condition} or {\em OM-cycle condition} as an alternative way to characterize OM-reversible Markov chains. The OM-cycle condition only requires knowledge of the orders of magnitude of the transition probabilities in each cycle in the graph corresponding to the Markov chain. The OM-cycle condition gives a direct way to check OM-reversibility, since it does not require constructing a $\nu$ as in the case of OM-detailed balance.

\begin{definition} \label{def:omkolm}
If for every finite sequence of states $(x_1,x_2,\ldots,x_{n-1},x_n=x_1) \subset \Om$ such that $\prod_{i=1}^{n-1} P(x_i,x_{i+1})>0$, the following condition holds
 \begin{align}\label{omkolm}
 \sum_{i=1}^{n-1} \phi(x_i,x_{i+1}) = \sum_{i=1}^{n-1} \phi(x_{i+1},x_{i}),
 \end{align}
 then we say that $(X,\Om,P)$ satisfies {\em order of magnitude Kolmogorov cycle condition} or {\em OM-cycle condition}.

\end{definition}

\begin{theorem} \label{thm:omkolm}
Let $(X,\Om,P)$ be an irreducible Markov chain. $X$ is OM-reversible if and only if $X$ satisfies the OM-cycle condition. 
\end{theorem}

\begin{proof}
Suppose first that $X$ is OM-reversible, so that there exists a $\nu : \Om \to \Z$ which satisfies $\nu(u) +  \phi(u,v) = \nu(v) + \phi(v,u)$ for all $u,v \in \Om$ such that $P(u,v)>0$.  Let $(x_1,x_2,\ldots,x_{n-1},x_n=x_1) \subset \Om$ be such that $ \prod_{i=1}^{n-1} P(x_i,x_{i+1})>0$. Then OM-reversibility implies that $ \prod_{i=1}^{n-1} P(x_{i+1},x_{i})>0$ and 
\begin{align*}
& \sum_{i=1}^{n-1} \phi(x_i,x_{i+1}) - \sum_{i=1}^{n-1} \phi(x_{i+1},x_{i})  = \sum_{i=1}^{n-1} \left(\phi(x_i,x_{i+1}) - \phi(x_{i+1},x_{i}) \right) \\
& = \sum_{i=1}^{n-1}  [\nu(x_{i+1})-\nu(x_i)] = \nu(x_{n})-\nu(x_1) = 0.
\end{align*}

Conversely, suppose that for every finite sequence of states $(x_1,x_2,\ldots,x_{n-1},x_n=x_1) \subset \Om$ such that $\prod_{i=1}^{n-1} P(x_i,x_{i+1})>0$, the OM-cycle condition \eqref{omkolm} holds. Let $x,y \in \Om$. Since $X$ is irreducible, there exists at least one sequence of states $\gamma = (y_1,y_2,\ldots,y_{n-1}) \subset \Om$ such that $\prod_{i=0}^{n-1}P(y_i,y_{i+1})>0$ where $y_0:=y$ and $y_n:=x$. 
Define 
\bd
 \mu(x,y)=\mu(x,y;\gamma) := \sum_{i=0}^{n-1}\left[\phi(y_i,y_{i+1}) - \phi(y_{i+1},y_{i})\right].
 \ed
To see that $\mu$ is independent of the sequence $\gamma$, let $y_{2n}:=y$ and suppose there is another sequence  
\bd
\gamma'  := (y_1',y_2',\ldots,y_{n-2}',y_{n-1}') := (y_{2n-1},y_{2n-2},\ldots,y_{n+2},y_{n+1}) \subset \Om
\ed
 such that $\prod_{i=0}^{n-1}P(y_i',y_{i+1}')>0$. Then 
 
\begin{align*} 
& \mu(x,y;\gamma) - \mu(x,y; \gamma') \\
&= \sum_{i=0}^{n-1}\left[\phi(y_i,y_{i+1}) - \phi(y_{i+1},y_{i})\right] - \sum_{i=0}^{n-1}\left[\phi(y_{2n-i-1},y_{2n-i-2}) - \phi(y_{2n-i-2},y_{2n-i-1})\right]\\
&=\sum_{i=0}^{n-1}\left[\phi(y_i,y_{i+1}) +  \phi(y_{2n-i-2},y_{2n-i-1})\right] - \sum_{i=0}^{n-1}\left[\phi(y_{2n-i-1},y_{2n-i-2}) + \phi(y_{i+1},y_{i})\right]\\
&=\sum_{i=0}^{2n-1} \phi(y_i,y_{i+1}) - \sum_{i=0}^{2n-1}\phi(y_{i+1},y_{i}) =0.
\end{align*}
 where we used the OM-cycle condition \eqref{omkolm} in the last step. 

It is easy to check that for $x,y,z \in \Om$, $\mu(x,y) + \mu(y,z) = \mu(x,z)$ and $\mu(y,x) = -\mu(x,y)$. 

Fix a state $x_0 \in \Om$ and for $x \in \Om$, define 
\begin{align}\label{def:nu}
\nu(x) := \mu(x,x_0) - \min_y \mu(y,x_0).
\end{align}
To see that $\nu$ is well-defined, let $x_0' \neq x_0$. So 
\begin{align*}
&\mu(x,x_0') - \min_y \mu(y,x_0') \\
&=  \mu(x,x_0) + \mu(x_0,x_0')- \min_y (\mu(y,x_0)+ \mu(x_0,x_0'))\\
& = \mu(x,x_0) - \min_y \mu(y,x_0).
\end{align*}

Let $u,v \in \Om$ such that $P(u,v)>0$. Then 
\begin{align*}
& \nu(u) - \nu(v) = \mu(u,x_0) - \mu(v,x_0) \\
& = \mu(u,x_0) + \mu(x_0,v) = \mu(u,v) = \phi(v,u) - \phi(u,v).
\end{align*} 
which shows that $X$ is OM-reversible.
\end{proof}

\begin{remark}
Theorem \ref{thm:omkolm} shows that for a finite, irreducible Markov chain, we can either take order of magnitude Kolmogorov cycle condition \eqref{omkolm} or order of magnitude detailed balance \eqref{omreversible} as the definition of OM-reversibility. The irreducibility condition is without loss of generality, because for reducible Markov chains we can focus attention on just the communicating, closed subset of states.
\end{remark}

\subsection{Characterization of the graph associated with the transition matrix of an OM-reversible Markov chain} 
Now that we have established our main result on order of magnitude of the stationary distribution, we make an observation about the structure of the graph associated with the transition matrix $P^\ep$ of an irreducible, OM-reversible Markov chain. We recall that such a graph is obtained by taking the vertex set to be the set of states $\Om$ and the weighted edges to be the elements of the transition matrix, with the weight proportional to the probability of transition. In particular, a transition with probability zero corresponds to an edge with weight zero, or equivalently to no edge at all. In order to establish the connection between OM-reversibility and the structure of the graph, we partition the state space as follows. 

\begin{definition}
Let $\Om_n = \{x \in \Om | \vp(\pi(x))=n\}$
\end{definition}

Clearly $\cup_{n =0}^\infty \Om_i$ is a partition of $\Om$. Suppose we draw the graph associated to $P^\ep$ so that $n$ is the $y$-coordinate (height) of the elements of $\Om_n$. For $m,k \geq 0$, let $x \in \Om_m$ and $y \in \Om_{m+k}$ be such that $P(x,y)>0$. Then $\phi(x,y) - \phi(y,x) = \vp(\pi(y)) - \vp(\pi(x)) = k \geq 0$. In particular if $k >0$, then the probability of a transition from $x$ to $y$ is at least one order of magnitude smaller than the probability of a transition from $y$ to $x$. In other words, downward transitions are more likely than upward transitions. This is the geometric characterization of an OM-reversible Markov chain. We refer to Figure \ref{4,5} for a particular example of an OM-reversible Markov chain, where such a property is evident. 

Since $\Om_0$ is the set of states with positive probability in the limit $\ep \to 0$, when considering applications of OM-reversibility it is important to characterize $\Om_0$, the set of states on which $\nu$ (a function that satisfies OM-detailed balance) attains a minimum. 

%
%


\section{Application: Clustering and pole formation} \label{section:clustering}

The phenomenon of aggregation of particles to a single location, known as {\em pole formation}, is the motivation for defining and studying OM-reversibility. It is of great interest to determine the fundamental principles of pole formation because it occurs in a wide variety of biological systems. We refer to a process that results in pole formation as a clustering process. To give one specific instance of a cellular clustering process, yeast cells can sometimes develop a bud on the surface, which initiates growth of the yeast at the budding site. The bud formation itself is initiated when molecules of Cdc42, initially scattered across the surface of the cell or within the cell, start aggregating at one location \cite{goryachev2008dynamics}. This and many other such phenomena in cell biology \cite{altschuler2008spontaneous, butty2002positive, drubin1996origins, ebersbach2007exploration,gassama2006phosphatidylinositol, ozbudak2005system,shi2003hippocampal,wedlich2004robust,weiner2002ptdinsp3} led us to study models of interacting particle system Markov chains where pole formation occurs. We found that the key property that many such interacting particle systems shared was that of OM-reversibility. 

Thus clustering processes are a natural choice as the first example of OM-reversibility. Conversely, clustering processes provide evidence of the usefulness of the notion of OM-reversibility. While reversibility may be a rare property in macroscopic systems such as interacting particle systems, we argue that OM-reversibility is much more common. This is because OM-reversibility requires only a mild condition on the order of magnitude of the transition probabilities. 

In this section, we provide sufficient conditions on the transition matrix of a Markov chain for pole formation. We study processes where at each time step a single particle jumps from a vertex $x$ to an adjacent vertex $y$, the probability of this jump depends on the occupancy of the vertex $x$ and the occupancies of all the neighboring vertices of $x$ including that of $y$. Processes where the probability of a jump depends only on the vertex $x$, called {\em zero range interaction processes}, have been studied in \cite{liggett1973infinite, spitzer1969random, spitzer1970interaction}. A process where the probability of transition depends on the occupancy of the vertex $x$ and on the occupancy of the vertex $y$, but not on the occupancy of the other neighbors of $x$, was studied in \cite{reversiblechain}.

We define  clustering processes in sections \ref{subsec:network}, \ref{subsec:state}, and \ref{subsec:clust}; and in section \ref{sec:poleformation} we show the existence of an integer-valued $\nu$ satisfying OM-detailed balance, thus establishing OM-reversibility of clustering processes.


\subsection{Network structure} \label{subsec:network}

Recall that a {\em network} $N=(V,E)$ is a finite, undirected, connected graph with vertex set $V$ and edge set $E$. If there is an edge connecting the vertices $v_i$ and $v_j$ we will say that $v_i$ and $v_j$ are {\em adjacent} and write $v_i \sim v_j$ or occasionally $i \sim j$, meanwhile the edge itself will be denoted by the unordered pair $\{v_i,v_j\}$. In the rest of this article, we consider an underlying network $N=(V,E)$ with $|V|:=m$. A vertex can be occupied by multiple particles, $n$ being the total number of particles in the network.  We label the vertices of the network $v_1,\ldots,v_m$.


\subsection{State space} \label{subsec:state}

The state space of the Markov chain $X^\ep(t)$  consists of all possible configurations of the $n$ particles among the $m$ vertices $(v_1, \ldots, v_m)$ of the network.  More precisely, a {\em configuration} or a {\em state} is an ordered collection of non-negative integers $(x_1,x_2,\ldots,x_m)$ such that $0 \leq x_i \leq n$ and $\sum_{i=1}^m x_i=n$. Denote the set of all states by $\Omega:=\Omega_{n,m}$. 

\begin{example} \label{example4,3}
Consider the case of $m=4$ vertices and $n=3$ particles. $\Om_{3,4}$ consists of 4 permutations of $(1,1,1,0)$, 12 permutations of $(2,1,0,0)$ and 4 permutations of $(3,0,0,0)$. The total number of states is $\abs{\Om_{3,4}}=20$. 
\end{example}

\begin{definition}
For distinct integers $i, j \in \{ 1,\ldots,m\}$, if $x, y \in \Omega$ are such that $y_j=x_j+1$, $y_i=x_i-1$, $y_k=x_k$ for $k \in \{1,\ldots,m\}\setminus \{i,j\}$, we will write $y= x^{i,j}$. To make later definitions easier to write, we will allow $x^{i,i} = x$. Note that if $y= x^{i,j}$ then $x=y^{j,i}$.

\end{definition}

We refer to a stochastic process involving multiple particles on a network $N$ as an {\em interacting particle system}.

\subsection{Definition of clustering process} \label{subsec:clust}

\begin{definition} \label{def:clustprocess}
We define a {\em clustering process} to be an interacting particle Markov chain $(X^\ep,\Om,P^\ep)$ where $\Om$ consists of all configurations of $n$ particles on an arbitrary network $N$ and where for all $x$ such that $x_i \geq 1$, and for $v_i \sim v_j$ or for $i=j$, $P^\ep$ satisfies

 \begin{align} \label{clustprocess}
 \vp_\ep(P^\ep(x,x^{i,j}))  =  \left\{ \begin{array}{c c} 1 & \mbox{ if } x_j=0 \\ 0 & \mbox{ if } x_j >0 \end{array} \right. .
\end{align} 
\end{definition}

\subsection{Models of clustering process} \label{sec:models}

\subsubsection{Clustering tendency}

We define a function $f$ that we will refer to as {\em clustering tendency}.

\begin{definition}\label{definition:f}
Let $f: \R_{\geq 0} \times \Z_{\geq 0} \ra \R_{\geq 0}$  be such that 
\bd
\vp_\ep(f(\ep, x)) = \left\{ \begin{array}{c c} 1 & \mbox{ if } x=0 \\ 0 & \mbox{ if } x >0 \end{array} \right. .
\ed

\end{definition}

We give some examples of the clustering tendency $f$. 

\begin{example} \label{ex:clusttend}
In the following examples the diffusion strength $\ep$ is nondimensional in order to make the discussion about small $\ep$ meaningful. Some examples of the clustering tendency $f$ are
\been
\item $\ds f_0(\ep,p) = \left\{\begin{array}{l} \epsilon,  \text{ if } p=0 \\ 
                           \frac 1 n, \text{ if } p>0
                          \end{array}\right.$ (step function). 
\item $\ds f_1(\ep,p) = \frac p n + \ep$ (linear).
\item $\ds f_2(\ep,p) = \left(\frac p n\right)^2 + \ep$ (quadratic). 
\enen
\end{example}

At each time step, we pick a particle uniformly at random, and move it either to one of the adjoining vertices or return it to the original vertex $v_i$. Each of these jump events occurs with a probability that is proportional to the clustering tendency $f(\ep,x_j)$ where $v_j$ is the destination vertex. In Model 1, the probability of a jump depends on the origin vertex $v_i$ and the destination vertex $v_j$. In Model 2, the probability of a jump depends on the origin vertex $v_i$ and all its neighbors $v_k \sim v_i$ including the destination vertex $v_j$. For a fixed $\ep$, we let $f_\ep := f(\ep,\cdot)$.

We will define the probability of transition from the state $x$ to the state $x^{i,j}$ where $v_j \sim v_i$. $P(x,x)$ is then determined from $\sum_y P(x,y)=1$. 

\subsubsection{Model 1  -  Interaction between the origin and the destination site} 
The following example is a slightly modified version of the process studied by Joshi et al. \cite{reversiblechain}. 
The probability of a jump depends on both the origin vertex and the destination vertex. On a $d$-regular network, define for $j \neq i$,
\begin{align} \label{reverstrans}
P^\ep(x, x^{i,j}) & =
       {\ds \frac 1 d \frac{x_i}{n}  \frac{f_\ep(x_j)}{f_\ep(x_j)+ f_\ep(x_i)}} 
       \end{align}
       
Moreover, for $x,y \in \Om$ we define $P^0(x,y) := \lim_{\ep \ra 0} P^\ep(x,y)$. 

\subsubsection{Model 2  -  Interaction between the origin site and its neighbors including the destination site}
 We introduce an example of a clustering process where the probability of transition depends on the origin vertex, and all its neighbors including the destination vertex. On any network $N$, if $v_j \sim v_i$, then

\begin{align} \label{transprobs}
P^\ep(x, x^{i,j}) & =
       {\ds \frac{x_i}{n}  \frac{f_\ep(x_j)}{f_\ep(x_i) + \sum_{k \sim i}f_\ep(x_k)}}  
       \end{align}

For $x,y \in \Om$ we define $P^0(x,y) := \lim_{\ep \ra 0} P^\ep(x,y)$.

\begin{theorem}
The process defined by the transition matrix \eqref{transprobs} in Model 1 and in Model 2 is a clustering process. 
\end{theorem}

\begin{proof}
We prove the statement for Model 2, since the proof for Model 1 is quite similar.
If $x_i \geq 1$, then $0 \leq \vp_\ep(f_\ep(x_i) + \sum_{k \sim i}f_\ep(x_k)) \leq \vp_\ep(f_\ep(x_i))=0$. So  $\vp_\ep(f_\ep(x_i) + \sum_{k \sim i}f_\ep(x_k)) = 0$ and 
\bd 
\vp_\ep(P^\ep(x, x^{i,j})) = \vp_\ep(f_\ep(x_j)) = \left\{ \begin{array}{c c} 1 & \mbox{ if } x_j=0 \\ 0 & \mbox{ if } x_j >0 \end{array} \right. .
\ed 
This proves the theorem.
\end{proof}

\subsection{Structure of the graph associated with the transition matrix of a clustering process}
In this section we describe the structure of the graph associated with the transition matrix of the Markov chain for the clustering processes \eqref{clustprocess}. We consider the instance where the underlying network consists of $m=4$ vertices arranged in a circle. In other words, each vertex has precisely 2 neighbors. Moreover, there are $n=5$ particles. Even for this relatively simple example, the state space is quite large, $\abs{\Om} = 56$. We will use the dihedral symmetry of the underlying network to define a new, but closely related Markov chain. The state space $\bOm$ of the new Markov chain $\bX$ consists of equivalence classes of states. Two states are considered in the same equivalence class if they have the same neighborhood structure. In other words if one state is in the orbit of the other state under the action of the dihedral group $D_8$, then the two states are equivalent. For instance, one of the equivalence classes is $\{(1,1,3,0), (0,1,1,3), (3,0,1,1), (1,3,0,1), (0,3,1,1), (3,1,1,0), (1,1,0,3), (1,0,3,1) \}$. We will select an arbitrary representative of the equivalence class to denote the entire equivalence class when referring to a state in $\bOm$. The details on defining  the transition probabilities of the new Markov chain with this symmetry can be found in \cite{reversiblechain}.

The new Markov chain $\bX^\ep$  has $\abs{\bOm}=10$ states shown in Figure \ref{4,5}. We represent the Markov chain as a graph whose vertices are states and a directed edge from state $x \in \Om$ to state $y \in \Om$ represents $P(x,y) >0$. The edge is represented in red (dark) if $\phi(x,y) = 0$ and in green (light) if $\phi(x,y)=1$. We have drawn the graph associated with the Markov chain so that the vertical coordinate of the state $x$ is the order of magnitude of the probability of the state in the stationary distribution $\vp(\pi(x))$. Because $\bP$ is OM-reversible, all downward arrows are red (dark) (probability of transition is $\Theta(\ep^0)$) and all upward arrows are green (light) (probability of transition is $\Theta(\ep^1)$). 

For $\ep=0$, the absorbing states of $X^0$ are the 1-pole state, $(0,0,0,5)$, and the 2-pole states, $(0,1,0,4)$ and $(0,2,0,3)$, {\it i.e.} absorbing states of $X^0$ are the ones for which all outgoing edges are green (light) or $\Theta(\ep^1)$. We show in Theorem \ref{thm:clusteringstat} that  only the single pole state $(0,0,0,5)$ has positive probability in the stationary distribution $\pi^\ep$ as $\ep \ra 0^+$. 


\begin{figure}[h!]
\begin{center}
{\includegraphics[angle=0,scale=1]{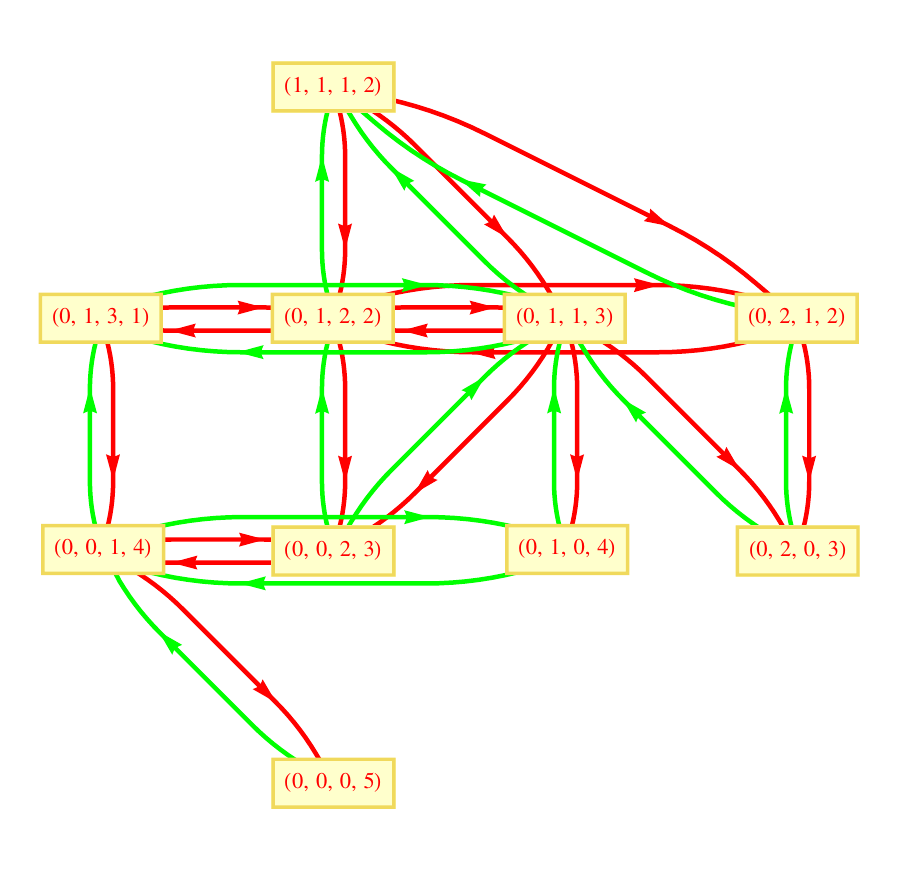} 
\label{4,5}}
\caption{Markov chain for $m=4$ bins and $n=5$ balls. The red (dark) arrows indicate that $\phi(x,y)=0$ while the green (light) arrows indicate that $\phi(x,y)=1$. The vertical coordinate of the state $x \in \Om$ is determined by $\vp(\pi(x))$. Of the three absorbing states $(0,0,0,5), (0,1,0,4)$ and $(0,2,0,3)$ of the reducible Markov chain $X^0$, only the single pole state $(0,0,0,5)$ has positive probability in the stationary distribution $\pi^\ep$ of the irreducible Markov chain $X^\ep$ as $\ep \ra 0^+$. The figure was drawn using {\tt Mathematica}.}
\end{center}
\end{figure}


\subsection{Pole formation in the clustering processes} \label{sec:poleformation}

In this section, we show that a clustering process is OM-reversible, even though not reversible in general. We first establish that the clustering process is irreducible and aperiodic for positive $\ep$. 

\begin{theorem}\label{thm:irred}
For $\ep >0$, a clustering process $X^\ep(t)$ is irreducible and aperiodic. 
\end{theorem}

\begin{proof}
Irreducibility follows because the underlying graph $N$ is connected and for $\ep >0$ each neighboring vertex is accessible with positive probability. Aperiodicity is almost immediate since this only requires that there is an $x \in \Om$  such that $P(x,x)>0$, which is in fact true for all $x \in \Om$.
\end{proof}

\begin{remark}
We should emphasize that the Theorem \ref{thm:irred} is true only for $\ep >0$. For $\ep=0$, a clustering process is reducible and has multiple absorbing states. In fact for $\ep=0$, any state for which the occupied vertices are isolated is an absorbing state. In other words, if $x \in \Om$ is such that if $x_i>0, x_j>0$ implies that $v_i \nsim v_j$ then $x$ is an absorbing state. 
\end{remark}

Since for $\ep >0$, a clustering process $X^\ep$ is irreducible and aperiodic, there exists a unique stationary distribution $\pi = \pi^\ep$ with $\pi^\ep(x) >0$ for all $x \in \Om$.  We now give a simple characterization of the probability of a given configuration in the stationary distribution in terms of the number of particles at each vertex.

\begin{definition} \label{def:support}
The {\em support of the state $x \in \Om$} is defined to be $s(x):=\{ v_i \in V | x_i \geq 1\}$. The number of occupied vertices or {\em the size of the support of the state $x \in \Om$} is the cardinality of the set $s(x)$, denoted by $\abs{s(x)}$.
\end{definition}

\begin{lemma}
For a clustering process, let $x,y \in \Om$ be such that $P^\ep(x,y)>0$. Then $\phi(x,y)=0$ if and only if $s(y) \subset s(x)$.  
\end{lemma}

\begin{proof}
If $\phi(x,y)=0$ then either $y=x$ or $y=x^{i,j}$ with $x_j>0$. In either case, $s(y) \subset s(x)$. On the other hand, if $\phi(x,y)=1$ then $y=x^{i,j}$ with $x_j=0$, so that $s(y) \not \subset s(x)$. 
\end{proof}

\begin{corollary} \label{cor:phivalues}
Let $x,y \in \Om$ with $P^\ep(x,y) > 0$.
\been
\item If $\phi(x,y)=0$ and $\phi(y,x)=0$ then $s(y)=s(x)$. 
\item If $\phi(x,y)=0$ and $\phi(y,x)=1$ then $s(y) \varsubsetneqq s(x)$ and $\abs{s(y)} = \abs{s(x)}-1$. 
\item If $\phi(x,y)=1$ and $\phi(y,x)=1$ then $\abs{s(x)}=\abs{s(y)}$. 
\enen
\end{corollary}

\begin{proof}
If $\phi(x,y)=\phi(y,x)=0$ then $s(y) \subset s(x)$ and $s(x) \subset s(y)$, which shows case 1. If $\phi(x,y)=0$ and $\phi(y,x)=1$ then $s(y) \subset s(x)$ but $s(x) \not \subset s(y)$, so that $s(y) \varsubsetneqq s(x)$. Moreover, $\abs{s(y)} < \abs{s(x)}$ and since only one particle is moved $\abs{s(x)} \leq \abs{s(y)}+1$; combining the two inequalities we have $\abs{s(x)} = \abs{s(y)}+1$ which proves case 2. Finally, if $\phi(x,y)=\phi(y,x)=1$ then $\abs{s(y)} \geq \abs{s(x)}$ and $\abs{s(x)} \geq \abs{s(y)}$ which shows case 3. 
\end{proof}

\begin{theorem} \label{thm:clusteringstat}
A clustering process $X^\ep(t)$ is OM-reversible and for $x \in \Om$, 
\bd \vp(\pi(x)) = \abs{s(x)}-1. \ed 
\end{theorem}

\begin{proof}
We need to find a $\nu: \Om \to \Z$ that satisfies \eqref{omrevplus}. If such a $\nu$ exists then, $\vp(\pi) - \nu \equiv c$. For $x \in \Om$, let $\nu(x) = \abs{s(x)}-1$.  From Corollary \ref{cor:phivalues}
\begin{align*}
\nu(x) -\nu(y) = \abs{s(x)} - \abs{s(y)} = \left\{ \begin{array}{c c} 0 & ,\mbox{ if } \phi(y,x) - \phi(x,y) = 0 \\
1 & ,\mbox{ if } \phi(y,x) - \phi(x,y) = 1 \end{array} \right.
\end{align*}
In either case, $\nu(x) -\nu(y) = \phi(y,x) - \phi(x,y)$. Furthermore, $\min_{x \in \Om} \nu(x) = \min_{x \in \Om} \abs{s(x)} - 1 = 0$, when there is at least one particle in the network. This proves the theorem. 
\end{proof}

\begin{corollary} \label{cor:singlepole}
$\Om_0 = \{z \in \Om  ~:~  \abs{s(z)}=1\}$.
\end{corollary}

Corollary \ref{cor:singlepole} determines $\Om_0$ for the clustering processes to be the set of {\em single pole states}, {\it i.e.} states for which all the particles are accumulated at a single vertex. It is natural to ask the question, ``which other processes besides the clustering processes have the set of single pole states as the limiting stationary distribution $\Om_0$?" We define a generalization of clustering processes which provides the answer. 

\begin{definition}
An interacting particle system $X^\ep$ on a network, for which at most one particle jumps to an adjacent vertex at each time step, is called a {\em generalized clustering process} if the transition matrix satisfies the following
\been
\item For all $x,y$ such that $P(x,y)>0$, $\vp(P(x,y)) \in \{0,1\}$. 
\item If $\abs{s(x)} = \abs{s(y)}$, then $\vp(P(x,y)) = \vp(P(y,x))$.  
\item If $\abs{s(x)} = \abs{s(y)}-1$, then $\vp(P(x,y)) = 1$ and $\vp(P(y,x))=0$.  
\enen
\end{definition}

A clustering process is a generalized clustering process as is clear from Corollary \ref{cor:phivalues}. We give two examples of generalized clustering processes that are not clustering processes.

\begin{example}
\been
\item Suppose the transition matrix $P$ satisfies the condition that for all states $x, x^{i,j}$ such that $P(x, x^{i,j})>0$, $\vp(P(x,x^{i,j})) =0$ if and only if $x_i=1$. In other words, the only event that does not have probability $\Theta(\ep)$ is the event where a particle leaves an empty vertex in its wake. This process is a generalized clustering process but not a clustering process. 
\item A slight variant of a clustering process is one where the transition probabilities obey \eqref{clustprocess} in Definition \ref{def:clustprocess} with the exception that when $x_i =1$ and $x_j =0$, then $\vp(P(x,x^{i,j}))=0$. In other words, a particle moves to an unoccupied vertex with probability $\Theta(\ep)$ unless it leaves an empty vertex in its wake, in which case the probability of transition is $\Theta(1)$. 

An explicit example of such a process is obtained as follows. On a $d$-regular network, define for $j \neq i$,
\begin{align} \label{reverstranswithneg}
P^\ep(x, x^{i,j}) & =
       {\ds \frac 1 d \frac{x_i}{n}  \frac{f_\ep(x_j)}{f_\ep(x_j)+ f_\ep(x_i-1)}} 
       \end{align}
       
An interpretation of the extra `$-1$' is that once a particle is picked, the probability of return to the vertex of origin depends only on the number of particles remaining. This process is defined and analyzed  in \cite{reversiblechain}, where it is shown to be reversible. Obviously, the process is OM-reversible.  Due to the extra `$-1$', the process is not a clustering process but it is a generalized clustering process because when $x_i =1$ and $x_j =0$, then $\vp(P(x,x^{i,j})=0$.

\enen
\end{example}

\begin{theorem} \label{thm:neccsuff}
A generalized clustering process is OM-reversible and the stationary distribution has order of magnitude $\abs{s(x)}-1$. Conversely, let $X^\ep$ be an irreducible, OM-reversible process on a finite state space such that for all $x,y \in \Om$, $\phi(x,y) \in \{0,1\}$ and such that the stationary distribution has order of magnitude $\abs{s(x)}-1$. Then $X^\ep$ is a generalized clustering process.
\end{theorem}

\begin{proof}
For a generalized clustering process $\abs{s(y)} - \abs{s(x)} = \vp(P(x,y)) - \vp(P(y,x))$, which proves the first part of the claim. 

Conversely, OM-reversibility of $X^\ep$ implies that if $\abs{s(y)} - \abs{s(x)} = 1$ then $\vp(P(x,y)) - \vp(P(y,x)) =1$. Since $\vp(P) \in \{0,1\}$, $\vp(P(x,y))=1$ and $\vp(P(y,x))=0$. If $\abs{s(y)} - \abs{s(x)} = 0$ then  $\vp(P(x,y)) - \vp(P(y,x)) =0 $. This shows that $X^\ep$ is a generalized clustering process. 
\end{proof}


\section{Application: Clustering with a carrying capacity} \label{section:carrcap}

In this section, we consider an extension of the clustering processes, namely clustering processes with a `soft' carrying capacity. We assume the same network structure, a connected, undirected graph, with $n$ particles initially distributed among the vertices. We assume that the particles have a tendency to cluster except when there are too few particles at a vertex or when there are too many particles at a vertex. For the vertex $v_i$, if the occupancy is under $L_i$, or over the carrying capacity $K_i$, a particle can arrive at $v_i$ only with probability that is $\Theta(\ep)$.

\begin{definition} \label{def:clustwithcarr}
Let $0 < L_j < K_j $ for all $j$. We define the {\em clustering process with carrying capacity} to be an interacting particle system $(X^\ep,\Om,P^\ep)$ where the state space $\Om$ consists of all configurations of $n$ particles on a connected network and where for $v_i \sim v_j$ or for $i=j$, $P^\ep$ satisfies

 \begin{align}
 \vp_\ep(P^\ep(x,x^{i,j}))  =  \left\{ \begin{array}{l l} 1 &, \mbox{ if } x_j \leq L_j \mbox{ or } x_j \geq K_j\\ 0 &, \mbox{ if } L_j< x_j <K_j \end{array} \right. .
\end{align} 
The ordered pair $(L_j,K_j)$ will be referred to as {\em the carrying capacity} of the vertex $v_j$.
\end{definition}


\subsection{Examples of clustering with carrying capacity}

We consider the simplest case where each vertex has the carrying capacity $(L,K)$, in other words, $L_j =L$ and $K_j=K$ for all $v_j \in V$. 

\begin{definition}\label{definition:fcarr}
Define the {\em clustering tendency} $f: \R_{\geq 0} \times \Z_{\geq 0} \ra \R_{\geq 0}$  to be

\bd
\vp_\ep(f(\ep, x)) = \left\{ \begin{array}{l l} 1 &, \mbox{ if } x \leq L \mbox{ or } x \geq K \\ 0 &, \mbox{ if } L < x < K \end{array} \right. .
\ed

\end{definition}

As before we think of $\ep$ as {\em diffusion}. We present some examples of the clustering tendency $f$.

\begin{example}
Once again $\ep$ is assumed to be nondimensional. 
\been
\item For $0 \leq L< K \leq n$, $f_0(\ep,p)= \left\{ \begin{array}{l l} \frac{1}{K-L} &, \mbox{ if } p \in (L,K) \\ \ep &, \mbox{ otherwise } \end{array}\right.$. 
\item For $0 \leq L< K \leq n$, $f_1(\ep,p)= \left\{ \begin{array}{l l} \frac{p-L}{K-L} + \ep&, \mbox{ if } p \in (L,K) \\ \ep &, \mbox{ otherwise } \end{array}\right.$. 
\item For $0 \leq L< K \leq n$, $f_2(\ep,p)=\max \left\{\frac{(p-L)(K-p)}{(K-L)^2},0 \right\} + \ep$. 
\enen
\end{example}

Let the transition matrix $P^\ep$ be as defined in Model 1 (equation \eqref{reverstrans}) or as in Model 2 (equation \eqref{transprobs}).  Then it is easy to check that $P^\ep$ is a clustering process with carrying capacity.


\subsection{Stationary distribution in clustering with carrying capacity}

\begin{theorem} \label{thm:carrstat}
If $\pi$ is the stationary distribution for a clustering process with carrying capacity $(L,K)$ for all vertices $v \in V$, then  
\begin{align*}
\vp(\pi(x)) &= \sum_{l=1}^m \left[\min(x_l,L) + \max(x_l-K+1,0) \right]  - c
\end{align*}
where $\ds c = \min_{x \in \Om}\left \{\sum_{l=1}^m \left[\min(x_l,L) + \max(x_l-K+1,0) \right] \right\}$.
\end{theorem}

\begin{proof}
Let $\mu(x) = \sum_{l=1}^m \left[\min(x_l,L) + \max(x_l-K+1,0) \right]$. We will show that $\mu$ satisfies OM-reversibility.
\begin{align*}
&\mu(x) - \mu(x^{i,j})\\
 = &\sum_{l=1}^m \left[\min(x_l,L) + \max(x_l-K+1,0) \right] - \sum_{l=1}^m \left[\min(x^{i,j}_l,L) + \max(x^{i,j}_l-K+1,0) \right] \\
= &\left[\min(x_i,L) + \max(x_i-K+1,0) \right] + \left[\min(x_j,L) + \max(x_j-K+1,0) \right] \\
&-  \left[\min(x^{i,j}_i,L) + \max(x^{i,j}_i-K+1,0) \right] -  \left[\min(x^{i,j}_j,L) + \max(x^{i,j}_j-K+1,0) \right]\\
& = \tau_i + \tau_j
\end{align*}
where 
\begin{align*}
\tau_i & :=\left[\min(x_i,L) + \max(x_i-K+1,0) \right] -  \left[\min(x_i-1,L) + \max(x_i-K,0) \right]   \\
\tau_j & := \left[\min(x_j,L) + \max(x_j-K+1,0) \right] - \left[\min(x_j+1,L) + \max(x_j-K+2,0) \right].
\end{align*}

\begin{align*}
\tau_i & = \left\{ \begin{array}{ll} 1 &, \mbox{ if } x_i \leq L \mbox{ or } x_i \geq K \\
0 &, \mbox{ if } L < x_i < K \end{array} \right. \\
\mbox{ and } \\
\tau_j & = \left\{ \begin{array}{ll} -1 &, \mbox{ if } x_j \leq L-1 \mbox{ or } x_j \geq K-1 \\
0 &, \mbox{ if } L-1 < x_j < K-1 \end{array} \right.
\end{align*}

Moreover, the following relations are true for the transition probabilities.

\been
\item If $x_i \leq L$ or $x_i \geq K$ and $x_j \leq L-1$ or $x_j \geq K-1$, then $\phi(x^{i,j},x) = \phi(x,x^{i,j}) = 1$.  
\item If $x_i \leq L$ or $x_i \geq K$ and $L-1 < x_j < K-1$, then $  \phi(x^{i,j},x) - \phi(x,x^{i,j}) = 1-0=1$.  
\item If $L < x_i < K$ and $x_j \leq L-1$ or $x_j \geq K-1$, then $\phi(x^{i,j},x) - \phi(x,x^{i,j}) = 0-1=-1$.
\item If $L < x_i < K$ and $L-1 < x_j < K-1$, then $\phi(x^{i,j},x) = \phi(x,x^{i,j}) = 0$.  
\enen

In all the cases, we have $\mu(x) - \mu(x^{i,j}) = \phi(x^{i,j},x) - \phi(x,x^{i,j})$, which shows that $\mu$ satisfies OM-reversibility. So that $\vp(\pi(x)) = \mu(x) - \min_y \mu(y)$. 
\end{proof}

\begin{theorem} \label{thm:omeganot}
Let $\mu(x)$ be as defined in the proof of Theorem \ref{thm:carrstat}. Let $z \in \Om_0$. Then the following statements hold:
\been
\item If $m > \llf \frac n K \rrf$ then $\mu(z) = L \llf \frac n K \rrf + \min\left(L, n-K \llf \frac n K \rrf \right)$. Further,
\been
\item If $n-K \llf \frac n K \rrf \leq L$ then $z \in \Om_0$ if and only if $z$ is a configuration for which $\llf \frac n K \rrf$ vertices contain at least $K$ particles.
\item If $n-K \llf \frac n K \rrf > L$ then $z \in \Om_0$ if and only if $z$ is a configuration for which $\llf \frac n K \rrf$ vertices contain exactly $K$ particles and one vertex contains the remaining $n - K \llf \frac n K \rrf$ particles.
\enen
\item If $m \leq \llf \frac n K \rrf$ then $\mu(z) = n - m(K-L)$ and $z$ is a configuration such that all vertices contain at least $K$ particles.
\enen

\end{theorem}
 

\begin{proof}
All the cases are obtained by maximizing the number of particles that are between the lower threshold $L$ and the upper threshold $K$.
 \end{proof}

A clustering process is a special case of a clustering process with a carrying capacity for $K > n$ and $L=1$. We recover Theorem \ref{thm:clusteringstat} as a corollary of Theorem \ref{thm:carrstat}.

\begin{corollary}
For a clustering process, $\vp(\pi(x)) = \abs{s(x)}-1$.
\end{corollary}

\begin{proof}
For $K > n$, $\max(x_l-K+1,0)=0$ and so with $L=1$, $\mu(x)$ defined in the proof of Theorem \ref{thm:carrstat} is
\bd
\mu(x) = \sum_{l=1}^m \min(x_l,1)=\abs{s(x)}.
\ed
Furthermore, for $K>n$, $\left \lfloor \frac n K \right \rfloor =0$ and so by Theorem \ref{thm:omeganot}
\begin{align*}
 \min_y \mu(y) & =  L \left \lfloor \frac n K \right \rfloor + \min \left(L, n - K \left \lfloor \frac n K \right \rfloor \right)  \\
& = \min(1,n) = 1,
\end{align*}
as long as there is at least one particle in the network.
So $\vp(\pi(x)) = \mu(x) - \min_y \mu(y) = \abs{s(x)}-1$.
\end{proof}

\section{Numerical studies of a clustering process} \label{section:numsim}

Theorem \ref{thm:clusteringstat} suggests that for the clustering processes with small values of the rate of diffusion, particles should cluster to a single vertex in the network. We simulated a particular instance of a clustering process, for which the underlying network is a torus of dimensions $(20,12)$. Initially, half the vertices are occupied with one particle and the other half are empty; empty vertices alternate with the occupied ones. The transition matrix is given by equation \eqref{transprobs} in Model 2 with the clustering tendency $f_\ep(p) = (p/n)^2 + \ep$. We used a diffusion of $\ep = 0.0005$ and ran the simulation for a total of $1,000,000$ time steps. In Figure \ref{fig:torus0005} we show successive snapshots of one simulation taken at times 1000, 40000, 80000 and 300000. As time progresses, we see particles accumulating at one site.

\begin{figure}[H] 
	\includegraphics[width=0.46\textwidth]{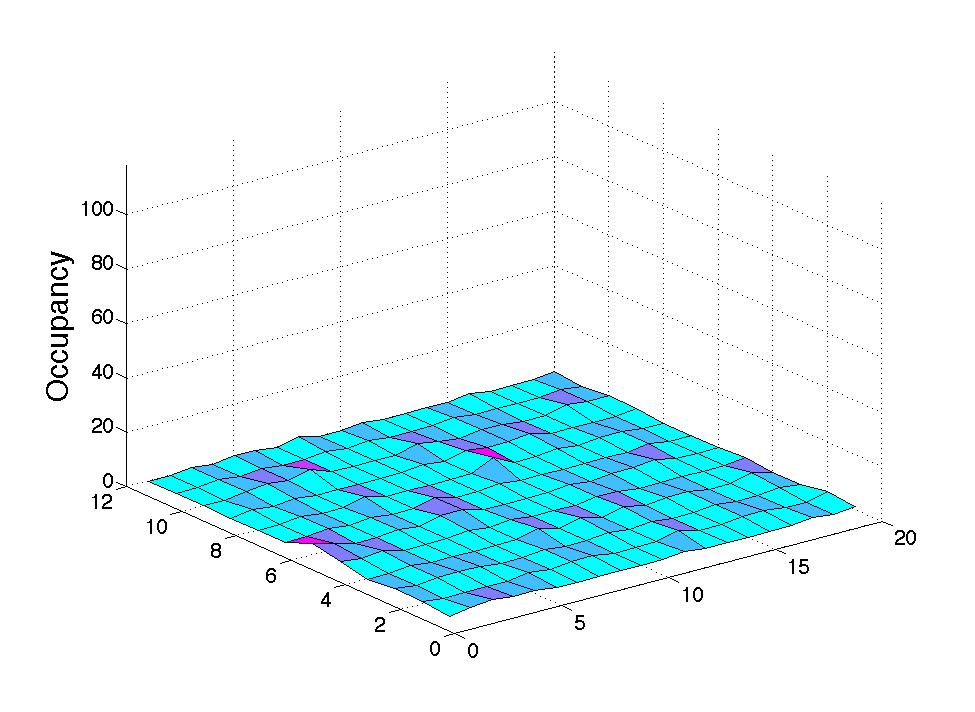}\centering
	\includegraphics[width=0.46\textwidth]{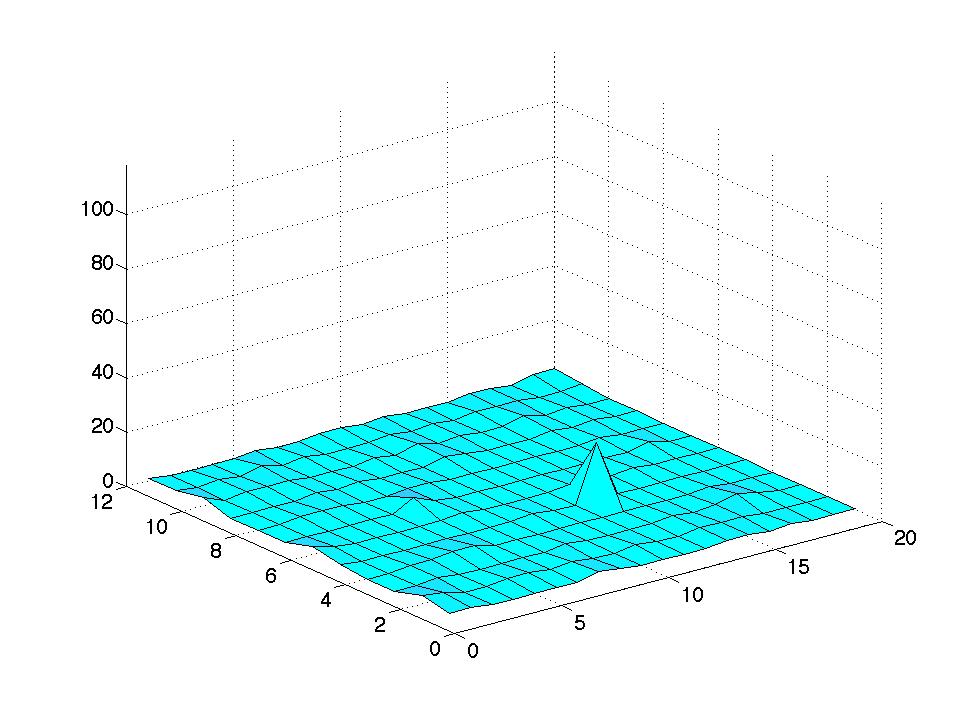}\\
	\includegraphics[width=0.46\textwidth]{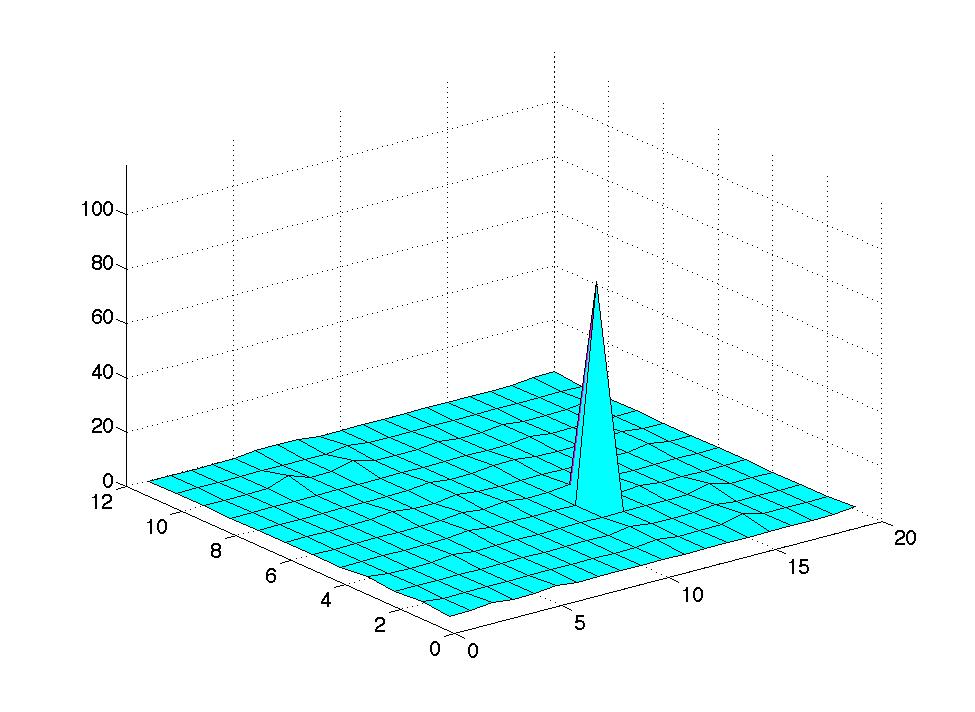} 
	\includegraphics[width=0.46\textwidth]{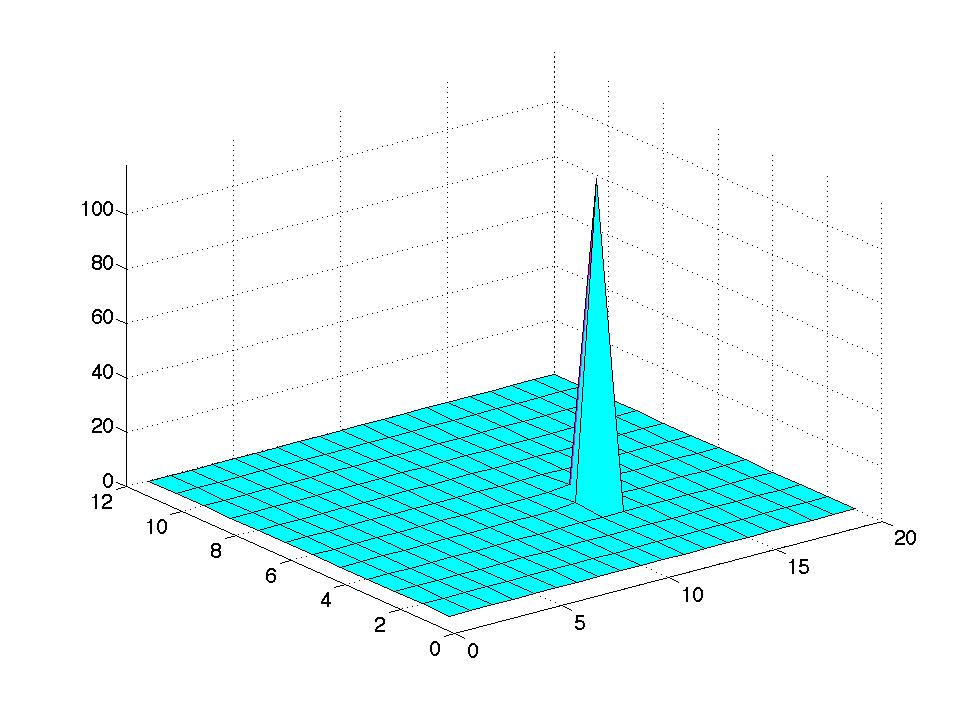}\centering
	\caption{We consider a torus of dimensions $(20,12)$, initially vertices containing one particle alternate with unoccupied vertices. The transition matrix is given by equations \eqref{transprobs} in Model 2 with the clustering tendency $f_\ep(p) = (p/n)^2 + \ep$, where $n=120$ and $\ep=0.0005$. The snapshots are taken at times 1000, 40000, 80000, and 300000. \label{fig:torus0005} }
\end{figure}

Define the {\em peak ratio $p(t)$} to be $\max_i \left\{\frac{x_i(t)}{n}\right\}$, the number of particles in the vertex with maximum occupancy divided by the total number of particles. In Figure \ref{fig:supppeak}, we plot the support size $\abs{s(X(t))}$ and the peak ratio $p(t)$. As clustering to a single pole takes place, the support size decreases to a value close to 1, while the peak ratio increases to a value close to 1. For the simulation corresponding to Figure \ref{fig:supppeak}, we calculated the average value of the support size $\abs{s(X)}$ over the last $500,000$ time steps and found that this quantity was $1.2219$. We calculated the average value of the peak ratio $p(t)$ over the last $500,000$ time steps and found that this quantity was $0.9981$. 

\begin{figure}[H] 
	\includegraphics[width=0.5\textwidth]{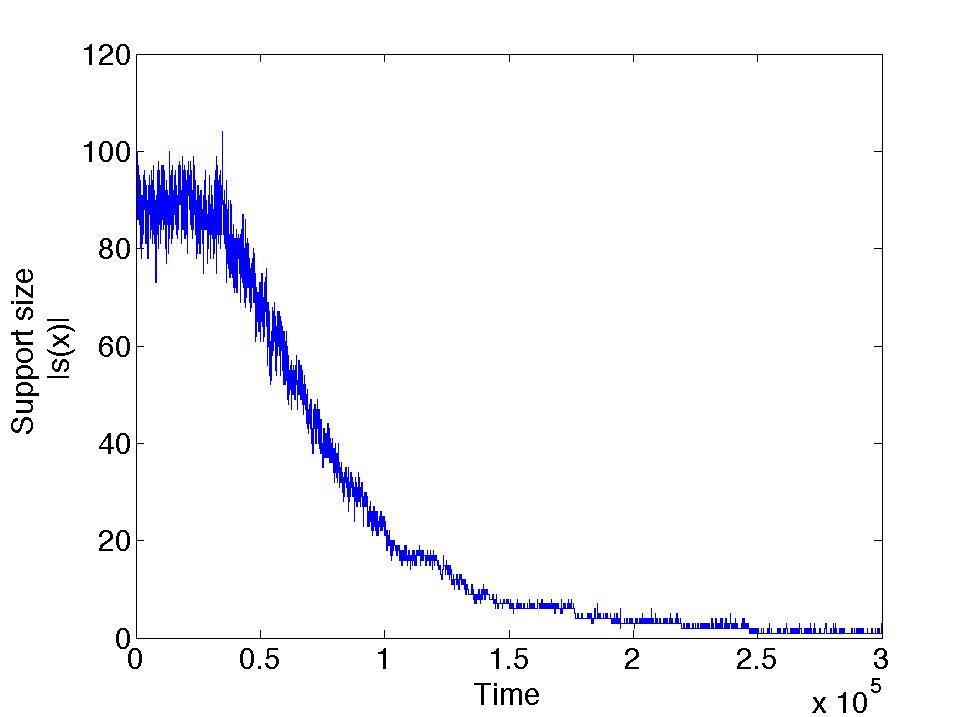}\centering
	\includegraphics[width=0.5\textwidth]{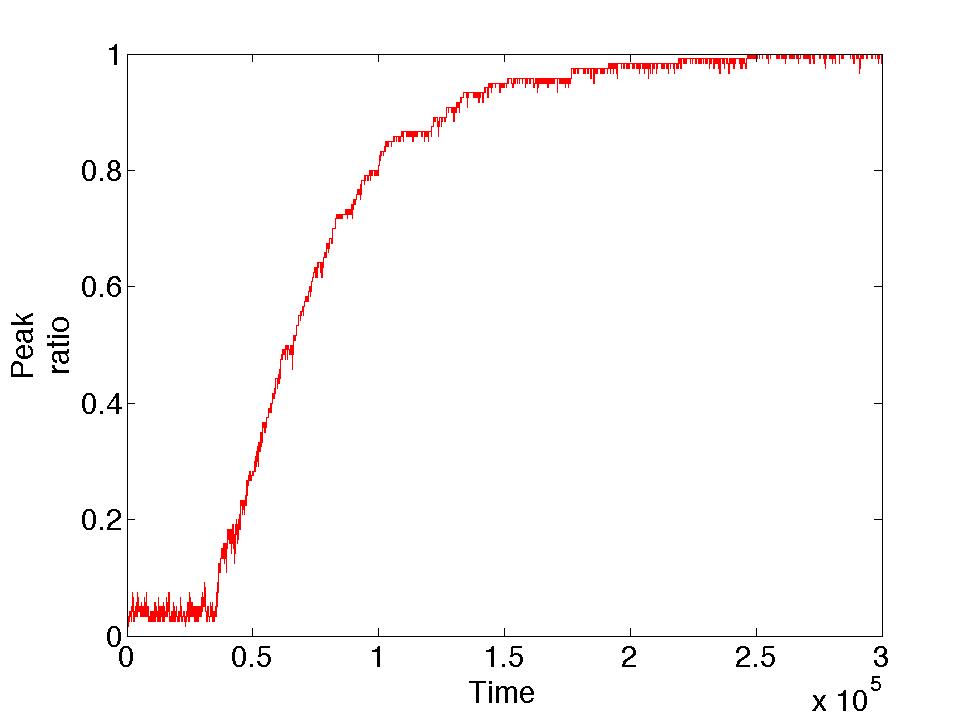}
	\caption{The support size $\abs{s(X)}$ approaches a value close to 1 as the particles cluster at one site, as depicted in the figure on the left. In the figure on the right we plot the peak ratio $p(t)$, which approaches a value close to 1. All the graphs in this section are plotted using {\tt Matlab}. \label{fig:supppeak} }
\end{figure}

\subsection*{Acknowledgments}
I am grateful to Rick Durrett, Michael C. Reed, and Scott McKinley for many helpful conversations and great advice. The author was partially supported by a National Science Foundation grant (EF-1038593). 
\bibliographystyle{amsplain}
\bibliography{bbs}

\providecommand{\bysame}{\leavevmode\hbox to3em{\hrulefill}\thinspace}
\providecommand{\MR}{\relax\ifhmode\unskip\space\fi MR }
\providecommand{\MRhref}[2]{%
  \href{http://www.ams.org/mathscinet-getitem?mr=#1}{#2}
}
\providecommand{\href}[2]{#2}
\begin{thebibliography}{10}

\bibitem{altschuler2008spontaneous}
S.J. Altschuler, S.B. Angenent, Y.~Wang, and L.F. Wu, \emph{On the spontaneous
  emergence of cell polarity}, Nature \textbf{454} (2008), no.~7206, 886.

\bibitem{avrachenkov2004first}
K.E. Avrachenkov and M.~Haviv, \emph{The first {L}aurent series coefficients
  for singularly perturbed stochastic matrices}, Linear algebra and its
  applications \textbf{386} (2004), 243--259.

\bibitem{butty2002positive}
A.C. Butty, N.~Perrinjaquet, A.~Petit, M.~Jaquenoud, J.E. Segall, K.~Hofmann,
  C.~Zwahlen, and M.~Peter, \emph{A positive feedback loop stabilizes the
  guanine-nucleotide exchange factor {C}dc24 at sites of polarization}, The
  EMBO journal \textbf{21} (2002), no.~7, 1565--1576.

\bibitem{casimir1949some}
H.B.G. Casimir, \emph{Some aspects of {O}nsager's theory of reciprocal
  relations in irreversible processes}, Il Nuovo Cimento (1943-1954) \textbf{6}
  (1949), 227--231.

\bibitem{drubin1996origins}
D.G. Drubin and W.J. Nelson, \emph{Origins of cell polarity.}, Cell \textbf{84}
  (1996), no.~3, 335.

\bibitem{durrett2010probability}
R.~Durrett, \emph{Probability: {T}heory and examples}, Cambridge University
  Press, 2010.

\bibitem{ebersbach2007exploration}
G.~Ebersbach and C.~Jacobs-Wagner, \emph{Exploration into the spatial and
  temporal mechanisms of bacterial polarity}, TRENDS in Microbiology
  \textbf{15} (2007), no.~3, 101--108.

\bibitem{feinberg1989necessary}
M.~Feinberg, \emph{Necessary and sufficient conditions for detailed balancing
  in mass action systems of arbitrary complexity}, Chemical Engineering Science
  \textbf{44} (1989), no.~9, 1819--1827.

\bibitem{gassama2006phosphatidylinositol}
A.~Gassama-Diagne, W.~Yu, M.~Ter~Beest, F.~Martin-Belmonte, A.~Kierbel,
  J.~Engel, and K.~Mostov, \emph{Phosphatidylinositol-3, 4, 5-trisphosphate
  regulates the formation of the basolateral plasma membrane in epithelial
  cells}, Nature cell biology \textbf{8} (2006), no.~9, 963--970.

\bibitem{gierer1972theory}
A.~Gierer and H.~Meinhardt, \emph{A theory of biological pattern formation},
  Biological Cybernetics \textbf{12} (1972), no.~1, 30--39.

\bibitem{goryachev2008dynamics}
A.B. Goryachev and A.V. Pokhilko, \emph{Dynamics of {C}dc42 network embodies a
  {T}uring-type mechanism of yeast cell polarity}, FEBS letters \textbf{582}
  (2008), no.~10, 1437--1443.

\bibitem{gupta2010stochastic}
A.~Gupta, \emph{Stochastic model for cell polarity}, Annals of Applied
  Probability (2011), In print, Available at arXiv:1003.1404.

\bibitem{hassin1992mean}
R.~Hassin and M.~Haviv, \emph{Mean passage times and nearly uncoupled {M}arkov
  chains}, SIAM Journal on Discrete Mathematics \textbf{5} (1992), no.~3,
  386--397.

\bibitem{irazoqui2003scaffold}
J.E. Irazoqui, A.S. Gladfelter, and D.J. Lew, \emph{Scaffold-mediated symmetry
  breaking by {C}dc42p}, Nature cell biology \textbf{5} (2003), no.~12,
  1062--1070.

\bibitem{reversiblechain}
Badal Joshi, Scott McKinley, Rick Durrett, and Michael~C. Reed, \emph{{A
  reversible {M}arkov chain as a model for symmetry-breaking and pole
  formation}}, In preparation, 2011.

\bibitem{kelly1979reversibility}
F.P. Kelly, \emph{Reversibility and stochastic networks}, Wiley, Chichester,
  1979.

\bibitem{lasserre1994formula}
J.B. Lasserre, \emph{A formula for singular perturbations of {M}arkov chains},
  Journal of applied probability \textbf{31} (1994), no.~3, 829--833.

\bibitem{latouche1978return}
G.~Latouche and G.~Louchard, \emph{Return times in nearly-completely
  decomposable stochastic processes}, Journal of Applied Probability
  \textbf{15} (1978), no.~2, 251--267.

\bibitem{lewis1925new}
G.N. Lewis, \emph{A new principle of equilibrium}, Proceedings of the National
  Academy of Sciences of the United States of America \textbf{11} (1925),
  no.~3, 179.

\bibitem{liggett1973infinite}
T.M. Liggett, \emph{An infinite particle system with zero range interactions},
  The Annals of Probability \textbf{1} (1973), no.~2, 240--253.

\bibitem{onsager1931reciprocal}
L.~Onsager, \emph{Reciprocal relations in irreversible processes. {I}.},
  Physical Review \textbf{37} (1931), no.~4, 405.

\bibitem{ozbudak2005system}
E.M. Ozbudak, A.~Becskei, and A.~van Oudenaarden, \emph{A system of
  counteracting feedback loops regulates {C}dc42p activity during spontaneous
  cell polarization}, Developmental cell \textbf{9} (2005), no.~4, 565--571.

\bibitem{schweitzer1968perturbation}
P.J. Schweitzer, \emph{Perturbation theory and finite {M}arkov chains}, Journal
  of Applied Probability \textbf{5} (1968), no.~2, 401--413.

\bibitem{shi2003hippocampal}
S.H. Shi, L.Y. Jan, and Y.N. Jan, \emph{Hippocampal neuronal polarity specified
  by spatially localized m{P}ar3/m{P}ar6 and {PI} 3-kinase activity}, Cell
  \textbf{112} (2003), no.~1, 63--75.

\bibitem{spitzer1969random}
F.~Spitzer, \emph{Random processes defined through the interaction of an
  infinite particle system}, Probability and Information Theory \textbf{89}
  (1969), 201--223.

\bibitem{spitzer1970interaction}
\bysame, \emph{Interaction of {M}arkov processes}, Advances in Mathematics
  \textbf{5} (1970), no.~2, 246--290.

\bibitem{turing1952chemical}
A.M. Turing, \emph{The chemical basis of morphogenesis}, Philosophical
  Transactions of the Royal Society of London. Series B, Biological Sciences
  \textbf{237} (1952), no.~641, 37.

\bibitem{wedlich2004robust}
R.~Wedlich-Soldner, S.C. Wai, T.~Schmidt, and R.~Li, \emph{Robust cell polarity
  is a dynamic state established by coupling transport and {GTP}ase signaling},
  The Journal of cell biology \textbf{166} (2004), no.~6, 889.

\bibitem{weiner2002ptdinsp3}
O.D. Weiner, P.O. Neilsen, G.D. Prestwich, M.W. Kirschner, L.C. Cantley, H.R.
  Bourne, et~al., \emph{A {P}td{I}ns{P}3-and {R}ho {GTP}ase-mediated positive
  feedback loop regulates neutrophil polarity}, Nature cell biology \textbf{4}
  (2002), no.~7, 509--513.

\bibitem{wigner1954derivations}
E.P. Wigner, \emph{Derivations of {O}nsager's reciprocal relations}, Journal of
  Chemical Physics \textbf{22} (1954), 1912--1915.

\end{thebibliography}

\end{document}